\documentclass[manuscript,12pt]{amsart}
\usepackage{amssymb}
\usepackage{verbatim}
\usepackage{graphicx}
\usepackage[cmtip,all]{xy}
\usepackage{amsmath}
\usepackage{amsfonts,amssymb}
\usepackage{times}
\usepackage{amscd}
\usepackage{epsfig}
\usepackage{float}
\usepackage{amsthm}
\usepackage{latexsym}
\usepackage{imakeidx}
\usepackage[active]{srcltx}

\newtheorem{theo}{Theorem}[section]
\newtheorem{thm}[theo]{Theorem}
\newtheorem{lem}[theo]{Lemma}
\newtheorem{cor}[theo]{Corollary}

\newtheorem{prop}[theo]{Proposition}

\newtheorem*{thmM}{Main Theorem}

\theoremstyle{definition}
\newtheorem{dfn}[theo]{Definition}

\theoremstyle{remark}

\numberwithin{equation}{section}

\newcommand{\crp}{\mathrm{CrP}}
\newcommand{\di}{\ol{\mathrm{Di}}}

\def\R{\mathbb{R}}
\def\Z{\mathbb{Z}}
\def\Mc{\mathcal{M}}

\newcommand{\oc}{\ol{c}}

\newcommand{\oy}{\ol{y}}
\newcommand{\oq}{\ol{q}}

\newcommand{\K}{\mathcal{K}}

\newcommand{\C}{\mathbb{C}}

\newcommand{\disk}{\mathbb{D}}
\newcommand{\cdisk}{\ol{\mathbb{D}}}

\newcommand{\ol}{\overline}

\newcommand{\sm}{\setminus}
\newcommand{\A}{\mathcal{A}}

\newcommand{\Tc}{\mathcal{T}}

\newcommand{\bd}{\mathrm{Bd}}

\newcommand{\lam}{\mathcal{L}}

\newcommand{\si}{\sigma}

\newcommand{\uc}{\mathbb{S}}

\newcommand{\al}{\alpha}

\newcommand{\be}{\beta}
\newcommand{\e}{\varepsilon}
\newcommand{\M}{\mathcal{M}}

\newcommand{\Cc}{\mathcal{C}}

\def\Sc{\mathcal{S}}
\def\Zc{\mathcal{Z}}
\def\cch{\mathrm{CCh}}

\renewcommand\le{\leqslant}
\renewcommand\ge{\geqslant}
\def\0{\varnothing}

\begin{document}

\date{November 22, 2020}

\title[A model of the cubic connectedness locus]{A model of the cubic connectedness
locus}

\author[A.~Blokh]{Alexander~Blokh}

\thanks{The second named author was partially
supported by NSF grant DMS--1807558}

\author[L.~Oversteegen]{Lex Oversteegen}

\author[V.~Timorin]{Vladlen~Timorin}

\thanks{The third named author has been supported by the HSE University Basic Research Program and Russian Academic Excellence Project '5-100'.}

\address[Alexander~Blokh and Lex~Oversteegen]
{Department of Mathematics\\ University of Alabama at Birmingham\\
Birmingham, AL 35294}

\address[Vladlen~Timorin]
{Faculty of Mathematics\\
HSE University\\
6 Usacheva str., Moscow, Russia, 119048}

\email[Alexander~Blokh]{ablokh@uab.edu}
\email[Lex~Oversteegen]{overstee@uab.edu}
\email[Vladlen~Timorin]{vtimorin@hse.ru}



\begin{abstract}

We construct a 
model of the cubic connectedness locus.


\end{abstract}

\maketitle

\section{Introduction}
The concepts of \emph{renormalization} and \emph{tuning}
appear in the context of the
Mandelbrot set $\Mc_2=\{c\in\C\mid Q_c^n(0)\not\to\infty\}$,
where $Q_c(z)=z^2+c$, and serve to explain a self-similar structure of $\M_2$.
\emph{Self-similarity} of $\Mc_2$ means, in particular, that there are infinitely many homeomorphic copies of $\Mc_2$ in $\Mc_2$,
 the so-called \emph{baby Mandelbrot sets}.
Baby Mandelbrot sets accumulate to any boundary point of $\Mc_2$.
If $c$ is in a baby Mandelbrot set, then $Q_c$ is obtained from another quadratic polynomial by ``tuning'',
i.e. consistently pinching closures of periodic Fatou components and their pullbacks.
A baby Mandelbrot set consists of all tunings of a given polynomial different from $z\mapsto z^2$ and is contained in a unique maximal one.

If we collapse the closure of the main cardioid and all maximal baby Mandelbrot sets, we will obtain a dendrite $D(\Mc_2)$
that reveals the \emph{macro}-structure of $\Mc_2$.
A self-similar description of $\M_2$ involves knowing $D(\Mc_2)$
together with a subset of marked points in $D(\Mc_2)$,
where each marked point is a collapsed maximal baby Mandelbrot set.

In this paper, we use these ideas to give a model of the cubic connectedness locus $\Mc_3$.
Basically, we consider, for every polynomial $f$, its \emph{chief} $g$,
that is, a polynomial such that $f$ is a tuning of $g$,
and $g$ is not a non-trivial tuning of another polynomial.
Our model for $\Mc_3$  relies upon studying chiefs of cubic polynomials.
For $d=2$, the chiefs define maximal baby Mandelbrot sets. Thus,
our approach gives a
macro-view of $\Mc_3$, similar to $D(\Mc_2)$.

We are not aware of any other models of $\M_3$.

\section{Statement of the main results}
The \emph{parameter space} of complex degree $d$ polynomials $P$ is the
space of affine conjugacy classes $[P]$ of these polynomials. The
\emph{connectedness locus} $\Mc_d$ consists of classes of all degree
$d$ polynomials $P$, whose Julia sets $J_P$ are connected. The
\emph{Mandelbrot set} $\Mc_2$ has a complicated self-similar structure
understood through the ``pinched disk'' model \cite{dh82,
hubbdoua85,thu85}.

In this paper, we find a combinatorially defined upper semi-continuous
(USC) partition of $\Mc_3$. A property of a polynomial is
\emph{combinatorial} if it can be stated based only upon knowing which
pairs of rational external rays land at the same point and which pairs
do not. A combinatorial USC partition of $\Mc_3$ yields a continuous
map of $\Mc_3$ to a quotient space of $\crp$, the space of unordered
cubic critical portraits. Let us describe our approach.

Let $P$ be such that $[P]\in \M_3$. A point $x$ is \emph{($P$-)legal}
if it eventually maps to a repelling periodic point.  An unordered pair
of rational angles $\{\al,\be\}$ is \emph{($P$-)legal} if the external
rays with arguments $\al$ and $\be$ land at the same legal point of
$P$. Write $\Zc_P$ for the set of all $P$-legal pairs of angles; call
$\Zc_P$ the \emph{l-set} of $P$. Such sets are closely related to
\emph{rational} and \emph{real laminations} of polynomials introduced
by J. Kiwi \cite{kiw01, kiwi97}.

A cubic polynomial $P\in \M_3$ is \emph{visible} if $\Zc_P\ne\0$ and
\emph{invisible} otherwise.
 If
$P$ is visible, denote by $\Cc_P$ the set of all critical portraits
\emph{compatible} with $\Zc_P$ (i.e., no critical chord from a critical
portrait in $\Cc_P$  separates a $P$-legal pair of angles). Clearly,
$\Cc_P$ is closed. The set $\Cc_P$ is \emph{the combinatorial
counterpart} of $P$.

In this paper we will define, for every $[P]\in\M_3$, a \emph{closed} subset
$\A_P$ of $\crp$ called an \emph{alliance}. The main properties of
alliances are:
\begin{enumerate}
  \item if $P$ is visible, then $\Cc_P\subset\A_P$;
  \item distinct alliances are disjoint;
  \item the alliances form an USC partition of $\crp$.
\end{enumerate}

One special alliance is said to be \emph{prime}. It contains $\Cc_P$
for all visible polynomials $P$ with a non-repelling fixed point, some
other combinatorial counterparts, and is associated with all invisible
polynomials $P$.
The other alliances are called \emph{regular}; they
are combinatorial counterparts of certain visible polynomials,
and there are uncountably many of them.

\begin{thmM}
All alliances form a USC partition $\{\A_P\}$ of $\crp$. The union of
regular alliances is open and dense in $\crp$. The map $P\mapsto \A_P$
is continuous and maps $\Mc_3$ onto the quotient space $\crp/\{\A_P\}$.
\end{thmM}


Thurston \cite{thu85} gave a detailed, 
conjecturally homeomorphic, model of the
quadratic connected locus $\M_2$. The situation with the cubic
connectedness locus $\M_3$ is different. Indeed, $\M_3$ is complex
2-dimensional. Cubic polynomials are richer dynamically than quadratic
ones (critical points are essential for the dynamics of polynomials,
and cubic polynomials generically have two critical points) which makes
the cubic case highly intricate combinatorially \cite{bopt14,bopt18}.
This results into a breakdown of some crucial steps of \cite{thu85}
(e.g., cubic invariant laminations admit wandering triangles
\cite{bo04, bowaga}). Also, $\M_3$ is known to be non-locally connected
\cite{la89} and to contain copies of various non-locally connected
quadratic Julia sets \cite{bh01}.
All this makes the cubic case much harder and significantly complicates a complete  description of $\M_3$.

\section{Critical portraits and laminations}\label{s:criplam}
We assume familiarity with complex polynomial dynamics, including Julia
sets, external rays, etc. All cubic polynomials in this paper are
assumed to be \emph{monic}, i.e., of the form $z^3+$ a quadratic
polynomial, and to have connected Julia sets. We can parameterize the
external rays of a cubic polynomial $P$ by \emph{angles}, i.e.,
elements of $\R/\Z$. The external ray of argument $\theta\in\R/\Z$ is
denoted by $R_P(\theta)$. Clearly, $P$ maps $R_P(\theta)$ to
$R_P(3\theta)$.

For sets $A$, $B$, let $A\vee B$ be the set of all unordered pairs
$\{a,b\}$ with $a\in A, b\in B$; thus, the \emph{l-set} $\Zc_P$ of $P$
consists of all pairs $\{\al,\be\}\in(\mathbb{Q}/\Z)\vee
(\mathbb{Q}/\Z)$ such that $R_P(\al)$ and $R_P(\be)$ land at the same
legal point of $P$.

A \emph{chord} $\ol{ab}$ is a closed segment connecting points $a, b$
of the unit circle $\uc=\{z\in\C\,|\,|z|=1\}$. If $a=b$, the chord is
\emph{degenerate}. Two \emph{distinct} chords of $\disk$ \emph{cross}
if they intersect in $\disk$ (alternatively, they are called
\emph{linked}). Sets of chords are \emph{compatible} if chords from
distinct sets do not cross.

Write $\si_d$ for the self-map of $\uc$ that takes $z$ to $z^d$. A
chord $\ol{ab}$ is ($\si_d$-) \emph{critical} if $\si_d(a)=\si_d(b)$.
Let $\cch$ be the set of all $\si_3$-critical chords with the natural
topology; $\cch$ is homeomorphic to $\uc$. A \emph{critical portrait}
is a pair $\{\oc,\oy\}\in\cch\vee\cch$ such that $\oc$ and $\oy$ do not
cross. Let $\crp$ be the space of all critical portraits. It is
homeomorphic to the M\"obius band \cite{tby20}.

Motivated by studying l-sets of visible polynomials, Thurston
\cite{thu85} defined \emph{invariant laminations} as families of chords
with certain dynamical properties. We use a slightly different approach
(see \cite{bmov13}).

\begin{dfn}[Laminations]\label{d:geolam}
A \emph{prelamination} is a family $\lam$ of chords
called \emph{leaves} such that distinct leaves are unlinked and all
points of $\uc$ are leaves. If, in addition,  the set $\lam^+=\bigcup_{\ell\in\lam}\ell$
is compact, then $\lam$ is called a
\emph{lamination}.
\end{dfn}

From now on $\lam$ denotes a lamination (unless we specify that it is a
pre\-la\-mi\-nation). \emph{Gaps} of $\lam$ are the closures of
components of $\disk\sm\lam^+$. A gap $G$ is \emph{countable $($finite,
uncountable$)$} if $G\cap\uc$ is countable and infinite (finite,
uncountable).
Uncountable gaps are called \emph{Fatou} gaps.
For a closed convex set $H\subset \C$, maximal  straight segments in $\bd(H)$ are called \emph{edges} of $H$.

Convergence of prelaminations $\lam_i$ to a set of chords $\mathcal E$
is 
understood as Hausdorff convergence of leaves of $\lam_i$ to
chords from $\mathcal E$.
Evidently, $\mathcal E$ is a lamination.
A lamination $\lam$ is \emph{nonempty} if it contains nondegenerate leaves, otherwise it is \emph{empty} (denoted $\lam_\0$).
Say that $\lam$ is \emph{countable}
if it has countably many nondegenerate leaves and \emph{uncountable}
otherwise; $\lam$ is \emph{perfect} if it has no isolated leaves (thus, Fatou gaps of perfect laminations have no critical edges).

If $G\subset\cdisk$ is the convex hull of $G\cap\uc$, define $\si_d(G)$
as the convex hull of $\si_d(G\cap\uc)$. \emph{Sibling
($\si_d$)-invariant laminations} modify Thurston's \cite{thu85}
invariant geodesic laminations. A \emph{sibling of a leaf $\ell$} is a
leaf $\ell'\ne \ell$ with $\si_d(\ell')=\si_d(\ell)$. Call a leaf
$\ell^*$ such that $\si_d(\ell^*)=\ell$ a \emph{pullback} of $\ell$.
Note that the map $\si_d$ can be extended continuously over $\lam^+$ by extending linearly over all leaves  of $\lam$.
We also denote this extended map by $\si_d$.

\begin{dfn}[\cite{bmov13}]\label{d:sibli}
A (pre)lamination $\lam$ is \emph{sibling ($\si_d$)-invariant} if
\begin{enumerate}
  \item for each $\ell\in\lam$, we have $\si_d(\ell)\in\lam$,
  \item for each $\ell\in\lam$ there exists $\ell^*\in\lam$ with $\si_d(\ell^*)=\ell$,
  \item for each $\ell\in\lam$ such that $\si_d(\ell)$ is a nondegenerate leaf,
  there exist $d$ \textbf{pairwise disjoint} leaves $\ell_1$, $\dots$, $\ell_d$ in $\lam$ such that
  $\ell_1=\ell$ and $\si_d(\ell_1)=\dots=\si_d(\ell_d)$.
\end{enumerate}
\end{dfn}

Collections of leaves from (3) above are \emph{full sibling
collections}. Their leaves cannot intersect even on $\uc$. By cubic
(resp., quadratic) laminations, we always mean sibling $\si_3$-(resp.,
$\si_2$-) invariant laminations. When dealing with cubic laminations,
we write $\si$ instead of $\si_3$. From now on $\lam$ (possibly with
sub- and superscripts) denotes a cubic sibling invariant
(pre)lamination.

These are properties of cubic sibling invariant laminations \cite{bmov13}:
\begin{description}
  \item[gap invariance] if $G$ is a gap of $\lam$, then $H=\si(G)$
      is a leaf of $\lam$ (possibly degenerate), or a gap of
      $\lam$, and in the latter case, the map $\si|_{\bd(G)}:\bd(G)\to
      \bd(H)$ is an orientation
      preserving composition of a monotone map and a covering map
      (gap invariance is a part of Thurston's original definition
      \cite{thu85});
  \item[compactness] if a sequence of sibling invariant
      prelaminations converges to a set of chords $\mathcal{A}$,
      then $\mathcal{A}$ is a sibling invariant lamination.
\end{description}

A chord $\ell$ is \emph{inside} a gap $G$ if $\ell$ is, except for the
endpoints, in the interior of $G$. A gap $G$ of $\lam$ is
\emph{critical} if either all edges of $G$ are critical, or there is a
critical chord inside $G$. A \emph{critical set} of $\lam$ is a
critical leaf or a critical gap.
We also define a \emph{lap} of $\lam$ as either a finite gap of $\lam$ or a nondegenerate leaf of $\lam$
 not on the boundary of a finite gap.

The following facts are well-known (see, e.g., \cite{bl02} or \cite{bopt20}).
Fatou gaps of $\si_d$-invariant laminations are (pre)periodic.
If $U$ is a $\si_d$-periodic Fatou gap of period $n$ and the map $\si_d^n|_{\bd(U)}$ is of degree $k>1$,
then there is a monotone map from $\bd(U)$ to $\uc$ collapsing edges of $U$ and semiconjugating
$\si_d^n|_{\bd(U)}$ with $\si_k$.
Call Fatou gaps from the cycle of $U$ \emph{periodic Fatou gaps of degree $k$}.
If now $\si_d^n|_{\bd(U)}$ is of degree $1$, then the monotone map collapsing edges of $U$ can be chosen to
semiconjugate $\si_d^n|_{\bd(U)}$ and an irrational rotation of $\uc$.
In this case, $U$ is a \emph{Siegel gap}.
In any cycle of Siegel gaps some will have critical edges.
In general, if $\ell$ is a critical edge of a Fatou gap, then $\ell$ is isolated.

\begin{lem}\label{l:limleaf}
Suppose that $\lam_i\to \lam$ are $\si_d$-invariant laminations, and let $G$ be a periodic lap of $\lam$.
Then $G$ is also a lap of $\lam_i$ for all sufficiently large $i$.
\end{lem}

\begin{proof}
Let $G$ be a lap and $\ell$ an edge of $G$; we write $k$ for the
minimal period of $\ell$. Then $\lam_i$, for large $i$, must have a lap
$G_i$ with $G_i\to G$. Choose an edge $\ell_i$ of $G_i$ so that
$\ell_i\to \ell$. Then $\ell_i$ does not cross $\ell$ for large $i$ as
otherwise the leaves $\si_d^k(\ell_i)$ and $\ell_i$ would cross.
Moreover, $\ell_i$ is disjoint from the interior of $G$ for large $i$
as otherwise $\si_d^k(\ell_i)$ would intersect the interior of $G_i$
(note that $\ell_i$ maps farther away from $\ell$ under $\si_d^k$). By
way of contradiction assume that $\lam_i$ do not contain $G$. Then
$G_i\supsetneqq G$ and $\ell_i\ne \ell$ for at least one edge $\ell$ of
$G$. It follows that $\si_d^k(G_i)\supsetneqq G_i$, a contradiction.
\end{proof}

Let us define  laminational analogs of the  sets $\Cc_P$.

\begin{dfn}\label{d:cc-lam}
For a given $\si_d$-invariant $\lam$, let $\Cc(\lam)$ be the family of all critical
portraits, compatible with $\lam$; if $\K\in \Cc(\lam)$ we say that
$\K$ is a critical portrait \emph{of} $\lam$.
\end{dfn}

A lamination $\lam$
is \emph{clean} if any pair of distinct non-disjoint leaves of $\lam$ is on the
boundary of a finite gap. Clean laminations give rise to equivalence
relations: $a\sim_\lam b$ if either $a=b$ or $a$, $b$ are in the same
lap of $\lam$. In that case the quotient $\uc/\sim_\lam=J_\lam$ is called a \emph{topological Julia set} and the  map $f_\lam:J_\lam\to J_\lam$, induced by $\si_d$,  is called a \emph{topological polynomial}.
By Lemma 3.16 of \cite{bopt16a}, any clean lamination
has the following property: if one endpoint of a
leaf is periodic, then the other endpoint is also periodic with the
same minimal period. Limits of clean $\si_d$-invariant laminations are
called \emph{limit laminations}, cf. \cite{bopt15} (e.g., clean laminations
are limit laminations).

\begin{dfn}[Perfect laminations \cite{bopt20}]\label{perfect}
The maximal perfect subset $\lam^p$ of $\lam$ is called the
\emph{perfect part} of $\lam$; a lamination $\lam$ is \emph{perfect} if
$\lam=\lam^p$.
\end{dfn}

Equivalently, one can define $\lam^p$ as the set of all leaves
$\ell\in \lam$ such that arbitrarily close to $\ell$ there are
uncountably many leaves of $\lam$. Evidently, perfect laminations are clean and, hence, limit laminations.

\begin{dfn}[Chiefs]
If $\lam$ is nonempty, a \emph{chief} of $\lam$ is defined
as a minimal,  by inclusion,  nonempty sublamination of $\lam$.
\end{dfn}

The next lemma follows from \cite{bopt20}.

\begin{lem}\label{l:perf-count}
The set $\lam^p$ is an invariant lamination. If $\lam$ is uncountable,
then $\lam^p\subset\lam$ is nonempty. A chief is perfect or
countable.
\end{lem}

\begin{proof} By Lemma 3.12 of \cite{bopt20},
the set $\lam^p$ is an invariant lamination. If $\lam$ is uncountable,
then by definition $\lam^p\subset\lam$ is nonempty. If a chief
$\lam$ is uncountable but not perfect, its perfect part
$\lam^p\subsetneqq \lam$ is a nonempty sublamination, a contradiction with the
assumption that $\lam$ is a chief.
\end{proof}

Given a chord $\ell=\ol{ab}$, let $|\ell|$ denote the length of the
smaller circle arc with endpoints $a$ and $b$ (computed with respect to the
Lebesgue measure on $\uc$ such that the total length of $\uc$ is 1);
call $|\ell|$ the \emph{length} of $\ell$.

\begin{lem}\label{l:compute}
Any nonempty lamination contains leaves of length $\ge \frac1{d+1}$.
\end{lem}

\begin{proof}
Indeed, for a nondegenerate leaf $\ell$ so that $|\ell|<\frac{1}{d+1}$, either $|\si_d(\ell)|=d|\ell|$ or $|\si_d(\ell)|\ge\frac1{d+1}$.
\end{proof}

\begin{lem}\label{l:simple}
If $\lam$ is nonempty, then $\lam$ contains a chief.
\end{lem}

\begin{proof}
Let $\lam_\al$ be a nested family of laminations.
Definition \ref{d:sibli} implies
that then $\bigcap \lam_\al$ is a sibling invariant lamination too.
If all $\lam_\al$ are nonempty, then by Lemma \ref{l:compute} each of them has a leaf of length at least $\frac{1}{d+1}$ and so
$\bigcap \lam_\al$ is nonempty. Now the desired statement follows
from the Zorn lemma.
\end{proof}

The next lemma follows from the definitions and the compactness
property of invariant laminations.

\begin{lem}\label{l:dense}
Let $\lam$ be a chief. If $\ell\in \lam$ is a nondegenerate
leaf, then the iterated pullbacks of the nondegenerate iterated images
of $\ell$ are dense in $\lam$.
\end{lem}

\section{Invariant gaps and prime portraits}
\label{s:prime} An \emph{invariant} gap is an invariant gap of a cubic lamination
$\lam$, not necessarily specifying $\lam$.
An infinite invariant gap is \emph{quadratic} if it has degree 2. By Section 3 of
\cite{bopt16a}, any quadratic invariant gap can be obtained as follows.
A critical chord $\oc$ gives rise to the complementary circle arc
$L(\oc)$ of length $2/3$ with the same endpoints as $\oc$. The set
$\Pi(\oc)$ of all points with orbits in $\ol{L(\oc)}$ is nonempty,
closed and forward invariant. Let $\Pi'(\oc)$ be the maximal perfect subset of $\Pi(\oc)$.
The convex hulls $G(\oc)$ of $\Pi(\oc)$  and  $G'(\oc)$ of $\Pi'(\oc)$ are
 invariant quadratic gaps, and any invariant quadratic gap is of one of these
forms. For any invariant gap $G$, finite or infinite, a \emph{major} of
$G$ is an edge $M=\ol{ab}$ of $G$, for which there is a critical chord
$\ol{ax}$ or $\ol{by}$ disjoint from the interior of $G$. By Section
4.3 of \cite{bopt16a}, a degree 1 invariant gap has one or two majors;
every edge of $G$ eventually maps to a major and if $G$ is infinite and
of degree 1, at least one of its majors is critical. An invariant gap
$G$ is \emph{rotational} if $\si=\si_3$ acts on its vertices (i.e.,
on $G\cap\uc$) as a combinatorial rotation. 
For
brevity say that a chord is \emph{compatible} with a finite collection
of gaps if it does not cross edges of these gaps.

For a critical chord $\ell$, let $I(\ell)$ be the complement of
$\ol{L(\ell)}$. Let $\K=\{\oc, \oy\}$ be a critical portrait. Call $\K$
\emph{weak} if the forward orbit of $\oc$ is disjoint from $I(\oy)$, or
the forward orbit of $\oy$ is disjoint from $I(\oc)$; otherwise call
$\K$ \emph{strong}.

\begin{lem}\label{l:weaklosed}
The set of strong critical portraits is open and dense in $\crp$ while
the set of weak critical portraits is closed and nowhere dense in
$\crp$.
\end{lem}

The proof of Lemma \ref{l:weaklosed} is left to the reader.

\begin{lem}\label{l:crit-must}
If $U$ is a degree one periodic infinite gap of $\si_d$ for some $d\ge 2$, then some image of $U$ has critical edges.
\end{lem}

\begin{proof}
It is well-known that every edge of $U$ eventually maps to a critical or a periodic leaf.
If gaps from the orbit of $U$ have no
critical edges, it must have some periodic edges. Let $\ell_1, \dots,
\ell_k$ be a maximal chain of concatenated edges of $U$ such that
$\si_d^n(\ell_i)=\ell_i, 1\le i\le k$. Then $\si_d^n$ restricted to a
small arc $I\subset \bd(U)$ adjacent to $L=\bigcup^k_{i=1} \ell_i$
repels points away from the appropriate endpoint of $L$. Since
$\si_d|_{\bd)U)}$ is of degree 1, then points of $I$ are attracted to a
$\si_3^n$-fixed point $x\in \bd(U)$. Since $\si_3$ is expanding on
$\uc$, this implies that a subarc of $\bd(U)$ must collapse to a point
under $\si_d^n$. Hence some image of $U$ has critical edges.
\end{proof}

Together with results of \cite{bopt16a}, Lemma \ref{l:crit-must}
implies Lemma \ref{l:prime-1}.

\begin{lem}\label{l:prime-1}
A critical portrait $\{\oc, \oy\}$ is compatible with an invariant
quadratic gap if and only if it is weak. Similarly, $\{\oc, \oy\}$ is
compatible with an infinite invariant gap if and only if it is weak.
\end{lem}

\begin{proof}
Suppose that $\K$ is weak and forward orbit $T$ of $\oy$ is disjoint
from $I(\oc)$; then $T\subset \Pi(\oc)$, and $\{\oc,\oy\}$ is
compatible with $G(\oc)$. Assume now that $\{\oc,\oy\}$ is compatible
with an infinite invariant gap $U$. If $U$ is quadratic, the claim
follows from the above given description. Let $U$ be of degree one and
let $\ell$ be a critical edge of $U$. Then either $\ell$ coincides
with, say, $\oy$, or $\oy,$ $\oc,$ and $\ell$ form a triangle. Thus, we
may assume that, say, $\oy$ is non-disjoint from $U$ and $T\subset
\bd(U)$. Since the degree of $U$ is one, $\oc$ is disjoint from the
interior of $U$. Clearly, $U\cap \uc$ cannot be contained in
$\ol{I(\oc)}$ as then there is no room for $\oy$ there;
hence $U\cap\uc\subset \ol{L(\oc)}$, we have $T\cap I(\oc)=\0$, and $\K$ is weak.
\end{proof}

Some laminations must have compatible weak critical portraits.

\begin{thm}\label{t:standard}
Suppose that a nonempty cubic $\lam$ has an infinite periodic gap $U$ and either $\si(U)=U$ or
$U$ shares an edge with a finite rotational lap of $\lam$.
Then there is a weak critical portrait compatible with $\lam$. 
\end{thm}

\begin{proof}
By Lemma \ref{l:prime-1} we may assume that $U$ shares an edge $\ell$
with a finite rotational lap of $\lam$. Choose a critical chord $\oc$
in a gap from the orbit of $U$ with an endpoint coinciding with an
endpoint of the corresponding image of $\ell$. Choose a critical chord
$\oy\ne \oc$ compatible with $\lam$ and not crossing $\oc$. Then
$\{\oc,\oy\}$ is compatible with a quadratic invariant gap and, by Lemma
\ref{l:prime-1}, $\{\oc,\oy\}$ is weak.
\end{proof}

\emph{From now on, by a chief, we mean a chief of
some \textbf{limit} lamination.} A gap $G$ of a lamination is
\emph{invariant} if $\si(G)=G$ (with ``$=$'' rather than
``$\subset$'').

\begin{lem}\label{l:inv-exist}
Any chief has an invariant lap or an invariant infinite gap.
\end{lem}

\begin{proof}
By Lemma 3.7 of \cite{gm93}, any clean lamination has a lap or an infinite gap $G$ such that $\si(G)=G$.
Passing to the limit, we conclude that any limit lamination has the same property
(even though the limit of finite invariant gaps may be infinite).
A chief of a limit lamination must then also have the above mentioned property.
\end{proof}

\begin{dfn}[Friends, prime critical portraits]\label{d:f-c-p}
Critical portraits $\K_1$, $\K_2\in\crp$ are \emph{friends} (through a
lamination $\lam$) if $\K_1, \K_2\in \Cc(\lam)$. A critical portrait
$\K$ is \emph{prime} if a friend of $\K$ has a weak friend. A
lamination is \emph{prime} if it has a prime critical portrait.
\end{dfn}

The next lemma follows from definitions and compactness of the family
of all cubic sibling invariant laminations.

\begin{lem}\label{l:closed}
If $\K_1$, $\K_2$ are friends, then they are compatible with some chief
$\lam$ (thus, $\K_1, \K_2\in \Cc(\lam)$). Friendship is a closed
relation: if $\K_i\to\K$ and $\K'_i\to \K'$ and $\K_i$ and $\K_i'$ are
friends for all $i$, then so are $\K$ and $\K'$.
\end{lem}

\begin{proof}
The former claim follows from definitions. To prove the latter claim, choose
laminations $\lam_i$ such that $\K_i, \K'_i\in \Cc(\lam_i)$.
Passing to a subsequence, arrange that $\lam_i\to \lam$. 
It follows that $\K$, $\K'\in \Cc(\lam)$ and, hence, $\K$ and $\K'$ are friends.
Observe that
$\lam$ is non-empty because by Lemma \ref{l:compute} each lamination $\lam_i$
contains a leaf of length at least $\frac{1}{d+1}$.
\end{proof}

Given $\lam$ and a nondegenerate leaf $\ell\in \lam$, let $\mathcal
G(\ell)$ be the set of iterated pullbacks of the nondegenerate iterated
images of $\ell$.

\begin{lem}\label{l:count1}
Let $\lam$ be a countable chief. Then:
\begin{enumerate}
  \item for any nondegenerate leaf $\ell\in \lam$, the set of all
      nondegenerate leaves in $\lam$ coincides with $\mathcal
      G(\ell)$;
  \item all nondegenerate leaves of $\lam$ are isolated;
  \item there is a weak critical portrait in $\Cc(\lam)$.
\end{enumerate}
\end{lem}

\begin{proof}
(1) Choose an isolated leaf $\ell_0\in \lam$. We claim that
$\mathcal G(\ell_0)$ is the set of all nondegenerate leaves of $\lam$.
Indeed, otherwise
choose a nondegenerate leaf $\ell\in \lam\sm \mathcal G(\ell_0)$; then
leaves of $\mathcal G(\ell)$ cannot approximate $\ell_0$ or coincide
with $\ell_0$, a contradiction with Lemma \ref{l:dense}. If now $\ell$ is
a nondegenerate leaf of $\lam$, then
$\ell\in\mathcal{G}(\ell_0)$, and hence
$\mathcal{G}(\ell)=\mathcal{G}(\ell_0)$ is the set of all nondegenerate
leaves of $\lam$, as desired.

(2) All non-isolated leaves in $\lam$ form a forward invariant closed
family of leaves. If $\ell$ is non-isolated, choose leaves $\ell_i\to
\ell$, choose their pullbacks $\oq_i$, and choose a converging
subsequence of these pullbacks; in the end we will find a non-isolated
leaf $\oq$ with $\si(\oq)=\ell$. Now, let $\ell$ be non-isolated and
non-critical.
Choose a sequence $\ell_i\to \ell$ so that each $\ell_i$ is not on the boundary of a critical polygon.
Then the siblings $\ell'_i$, $\ell''_i$ of $\ell_i$ are well defined, and $\ell'_i\to \ell'$ while $\ell_i''\to \ell''$.
Clearly,
$\si(\ell)=\si(\ell')=\si(\ell'')$. We claim that $\ell$, $\ell'$ and
$\ell''$ are pairwise disjoint. Indeed, if, say, $\ell=\ol{ab}$ and
$\ell'=\ol{bc}$, where $\si(c)=\si(a)$, then $\ell_i$ and $\ell'_i$
have \emph{distinct} endpoints close to $b$ and mapping to the same
point; a contradiction. Hence by definition the set of all non-isolated
leaves of $\lam$ is itself a sibling-invariant lamination, a
contradiction with $\lam$ being a chief.

(3) By Lemma \ref{l:inv-exist}, we can find an invariant lap or
infinite gap $T$ of $\lam$. If $T$ is infinite, our claim follows from
Theorem \ref{t:standard}. Hence we may assume that $T$ is finite. Let
$\ell$ be an edge of $T$; it is isolated by (2). Let $H$ be a gap of
$\lam$ attached to $T$ along $\ell$. If $H$ is infinite, the desired
statement follows from Theorem \ref{t:standard}. Assume that $H$ is
finite. If $n$ is the minimal period of edges of $T$, then there are
two cases: $\si^n(H)=H$ and $\si^n(H)=\ell$. The former case
contradicts (1), hence $\si^n(H)=\ell$, and we may assume that
$\si(H)=\si(\ell)$.
Choose a critical chord $\oy\subset H$ that shares an
endpoint with $\ell$, and a critical chord $\oc$ in a critical gap or
leaf of $\lam$ disjoint from $H$.
(If $H$ has degree 3, then simply take a critical portrait $\{\oc,\oy\}$ in $H$.)
The critical portrait $\{\oc, \oy\}$ is compatible with $G(\oc)$, hence weak by Lemma \ref{l:prime-1},  as desired.
\end{proof}

The next definition complements Definition \ref{d:f-c-p}.

\begin{dfn}[Regular critical portraits, laminations, and chiefs]
A chief is \emph{regular} if all its critical portraits have only
strong friends. A lamination is \emph{regular} if it has a regular
chief. A critical portrait is \emph{regular} if it is compatible with a
regular chief.
\end{dfn}

Regular chiefs have nice properties.

\begin{lem}\label{l:oa-lam-1}
Suppose that $\lam$ is a regular chief. Then $\lam$ is perfect. Also,
$\Cc(\lam)$ is disjoint from any set $\Cc(\lam')$ where $\lam'\ne \lam$ is a
regular chief.
\end{lem}

\begin{proof}
By Lemma \ref{l:count1}, the lamination $\lam$ is uncountable; hence by Lemma
\ref{l:perf-count} it is perfect. Let $\K=\{\oc,\oy\}$ be a  critical
portrait compatible with $\lam$ and a chief $\lam'\ne \lam$. Since $\K$
is strong and has only strong friends, by Lemma \ref{l:prime-1}
invariant sets of $\lam$ and $\lam'$ are finite, and $\lam'$ is perfect by Lemma \ref{l:count1}.

Let $G$ and $G'$ be invariant laps of $\lam$ and $\lam'$, resp.,
located in the same component of $\cdisk\sm \bigcup\K$. 
No leaves of $\lam$ intersect the interior of $G'$ since otherwise
uncountably many leaves of $\lam$ would intersect edges of $G'$
contradicting the fact that $\K$ is compatible with $\lam$ and $\lam'$,
and \cite[Lemma 3.53]{bopt20}. Therefore, $G'\subset G$. Similarly,
$G\subset G'$, hence $G=G'$.

If iterated images of $\oc$ and $\oy$ avoid $G$, then iterated
$\lam$-pullbacks of $G$ and iterated $\lam'$-pullbacks of $G'$ are the
same. Hence $\lam=\lam'$ since the iterated pullbacks of $G$ are dense
in both $\lam$ and $\lam'$ by Lemma \ref{l:dense}. Let for some minimal
$n\ge 0$ the point $\si_3^n(\oc)$ be a vertex of $G$. Let $C$, $C'$ be
the critical sets of $\lam$, resp., $\lam'$ containing $\oc$. Since
infinite gaps of $\lam$ and $\lam'$ are disjoint from $G$ by Theorem
\ref{t:standard}, the sets $C$, $C'$ are laps.
Since $\lam$ and $\lam'$ are
compatible and perfect, $C=C'$ by \cite[Theorem 3.57]{bopt20}.
Similarly, we see that either $\oy$ never maps to $G$ or the critical
sets of $\lam$, $\lam'$ containing $\oy$ coincide. Thus, pullbacks of
$G$ in $\lam$ are the same as pullbacks of $G$ in $\lam'$, and
$\lam=\lam'$.
\end{proof}

Let us establish a few useful facts concerning regular chiefs.

\begin{lem}\label{l:periodense}
Suppose that $\lam$ is a regular chief. Then $\lam$ has infinitely many periodic
laps. Moreover, for any periodic leaf $\ell$ of $\lam$ the family of its pullbacks is
dense in $\lam$.
\end{lem}

\begin{proof}
Since $\lam$ is perfect, there are uncountably many grand orbits of non-degenerate
leaves of $\lam$ none of which contains a leaf from a critical set of $\lam$.
Choose a leaf $\ell$ from one of these grand orbits.
Since $\lam$ is perfect, it is clean.
Consider the associated topological polynomial $f_\lam:J_\lam\to J_\lam$ to which $\si_3|_{\uc}$
semiconjugate by a map, say, $\varphi$; then $\varphi(\ell)=x$ is a cutpoint of $J_\lam$ such that
all points of the $f_\sim$-orbit of $x$ are cutpoints of $J_\lam$. Such dynamics was studied in
\cite{bopt13} where, in Theorem 3.8, it was proven that then $f_\lam$ has infinitely many
periodic cutpoints. Taking their $\varphi$-preimages, we see that $\lam$ has infinitely many periodic
laps. Now, take a periodic leaf $\ell$ of $\lam$, consider its grand orbit and then its closure. By
the compactness of the family of laminations, this grand orbit is dense in $\lam$ as desired.
\end{proof}

Corollary \ref{c:friends-regular} follows from definitions and Lemma
\ref{l:oa-lam-1}.

\begin{cor}\label{c:friends-regular}
A friend of a regular critical portrait is regular. All critical
portraits of a regular chief $\lam$ form a closed subset $\Cc(\lam)$ of
$\crp$ consisting of friends, and no other critical portrait can be
their friend. A regular lamination has a unique chief that is regular
too.
\end{cor}

One can define regular critical portraits through the concept of a
friend.

\begin{lem}\label{l:cousin}
A critical portrait $\K$ is regular if and only if
all friends of friends of $\K$ are strong (i.e., if $\K$ is not prime).
\end{lem}

\begin{proof} Let $\K$ be a regular critical portrait. Then there is a
unique regular chief $\lam$ compatible with $\K$. All critical
portraits of $\lam$ are strong and have only strong friends; by Lemma
\ref{l:oa-lam-1}, none of them is compatible with a chief $\lam'\ne
\lam$. Hence all friends of $\K$ are compatible with $\lam$, i.e. are
regular. Repeating this, we see that friends of friends of $\K$ are
compatible with $\lam$ and, hence, strong. On the other hand, suppose
that all friends of friends of a critical portrait $\K$ are strong.
Take a chief $\lam$ compatible with $\K$. Then all its critical
portraits have only strong friends. By definition $\lam$  is regular
which implies that $\K$ is regular as desired.
\end{proof}

The terminology is partially self-evident and partially explained by
the fact that if $\lam$ is a dendritic regular chief then $J_\lam$ is a
dendrite.

\begin{lem}
  \label{l:prime-cl}
A limit of prime critical portraits is prime.
\end{lem}

\begin{proof}
If $\K_i$ are prime and $\K_i\to\K$, then, by definition, some friends
$\K'_i$ of $\K_i$ have weak friends $\K''_i$. Passing to a subsequence
we can arrange that $\K'_i\to\K'$ and $\K''_i\to\K''$.
By Lemma \ref{l:closed}, the portrait $\K'$ is a friend of $\K$, and $\K''$ is a friend of $\K'$.
By Lemma \ref{l:weaklosed}, the portrait $\K''$ is weak.
By definition, $\K$ is prime.
\end{proof}

Let us define alliances.

\begin{dfn}\label{l:alli}
The \emph{prime alliance} $\A_0$ is the set of all prime critical
portraits. A \emph{regular alliance} is the set $\Cc(\lam)$ where
$\lam$ is a regular chief.
\end{dfn}

We are ready to prove a part of the Main Theorem that can be viewed as its
combinatorial analog.

\begin{lem}\label{l:alli}
Alliances form a USC-partition of the set $\crp$.
The union of regular alliances is open and dense in $\crp$.
\end{lem}

\begin{proof}
That alliances form a partition of $\crp$ follows from definitions and
Lemma \ref{l:oa-lam-1}. Hence if two critical portraits are friends,
then they belong to the same alliance. Suppose that $\K_i\to \K$ and
$\K'_i\to \K'$ where $\K_i, \K'_i$ are friends. We may assume that
either all $\K_i, \K_i'$ are prime, or all $\K_i, \K_i'$ are strong. In
the former case Lemma \ref{l:prime-cl} implies that both $\K$ and $\K'$
are prime, and we are done.
In the latter case, $\K$ and $\K'$ are friends by Lemma \ref{l:closed}, and we are done too.
This proves the first claim.

To prove the second claim, observe that
the union of all regular alliances is open. Let $\K=\{\oc, \oy\}$ be a
critical portrait such that the orbits of $\si(\oc)$ and $\si(\oy)$ are
dense in $\uc$. Such portraits are dense in $\crp$. We prove that $\K$
is regular by proving that, for any friend $\K'=\{\oc',\oy'\}$ of $\K$,
the orbits of $\si(\oc')$ and $\si(\oy')$ are dense in $\uc$.

Let $\lam$ be a chief compatible with $\K$ and $\K'$. Let
$C$ be the leaf $\oc$ if $\oc\in \lam$ or the critical gap of $\lam$
containing $\oc$ otherwise. Define $Y$ similarly. Arrange that
$\oc'\subset C$ and $\oy'\subset Y$, possibly renaming $\oc'$ and
$\oy'$. We claim that the orbit of $\si(\oc')$ is dense in $\uc$.
Otherwise consider the nondegenerate chord $\oq=\ol{xx'}$, where $x=\si(\oc)$ and $x'=\si(\oc')$.
There is $\e>0$ and an arc
$I\subset\uc$ such that $\si^n(x')$ is never $\e$-close to $I$. On the
other hand, iterated images of $x$ are dense in $I$; the corresponding
images of $\oq$ have length $\ge\e$. Therefore, all leaves of $\lam$
originating in $I$ have length $\e$ or more.

Note that $C$ and $Y$ are not periodic, therefore, no $\si$-periodic
point of $\uc$ is an eventual image of $x$ or $x'$. There is a positive
integer $N$ with $\si^N(I)=\uc$. Since any $\si$-periodic point $a$ of
$\uc$ has a $\si^N$-preimage in $I$, we have $\ol{ab}\in\lam$ for some
$b\ne a$.  Since $\lam$ is perfect and hence clean,  endpoints of periodic leaves must have the same period.
Thus, the horizontal diameter $\di$ connecting the two
$\si$-fixed points of $\uc$ is a leaf of $\lam$. Consider a
nondegenerate chord $\ell$ with endpoint $i=e^{2\pi i(1/4)}$. Then
$\si^n(\ell)$ crosses $\oc$ or $\oy$ for some $n\ge 0$. Thus
$\ell\notin\lam$, a contradiction. We conclude that the orbit of
$\si(\oc')$ is dense. Similarly, the orbit of $\si(\oy')$ is dense.
\end{proof}

\section{The model}
Let $P$ be a polynomial with $[P]\in \M_3$.
Let $\sim_P$ be the equivalence
relation on $\uc$ defined by $e^{2\pi i\al}\sim_P e^{2\pi i\be}$ if
$\{\al,\be\}\in\Zc_P$ or $\al=\be$.
Let $\Zc^{lam}_P$ (from ``lamination'') be the set of all edges of the convex hulls in $\cdisk$ of all $\sim_P$-classes.
Define $\lam^r_P$ as the set of all edges of the convex hulls of all
$\sim_P$-classes and the limits of these edges. By the
compactness of laminations (see Section \ref{s:criplam}), $\lam^r_P$ is
a clean cubic lamination (cf. \cite{bmov13}).

The lamination $\lam^r_P$ is associated with an equivalence relation $\sim_{\lam^r_P}$
so that all laps of $\lam^r_P$ are convex hulls of $\sim_{\lam^r_P}$-classes (see discussion
in Section \ref{s:criplam} right after Definition \ref{d:cc-lam}). It is easy to see that
$\sim_{\lam^r_P}$ is the closure of $\sim_P$, so in what follows we simply denote it by $\sim_P$.

Suppose that $P$ has no neutral periodic points.
Then $\Zc^{lam}_P$ coincides with the \emph{rational} lamination \cite{kiw01} of $P$
 while $\lam^r_P$ coincides with the \emph{real} lamination \cite{kiwi97} of $P$.
 The next lemma summarizes some results of \cite{kiw01, kiwi97}.

\begin{lem}
  \label{l:gland}
Suppose that $P$ has no neutral cycles. Then the restriction $P|_{J_P}$ is monotonically
semiconjugate 
to the topological polynomial
$f_{\sim_P}:J_{\sim_P}\to J_{\sim_P}$ on its topological Julia set
so that fibers of this semiconjugacy are trivial at all (pre)periodic points
of $J_{\sim_P}$ (thus, for a periodic lap
$G$ of $\lam^r_P$, external rays corresponding to vertices of $G$ land
at the same legal point). Also, any clean lamination without infinite periodic
degree 1 gaps has the form $\lam^r_P$ for some $P$.
\end{lem}

If $\lam^r_P$ is regular, then
by Corollary \ref{c:friends-regular} it has a unique regular chief
denoted by $\lam^c_P$; set $\A_P=\Cc(\lam^c_P)$ (observe that then
$\A_P$ is a regular alliance). Equivalently, $\A_P$ can be defined as
the set of all friends of critical portraits from $\Cc_P$. However, if
$\lam^r_P$ is prime (e.g., if $P$ is invisible and, hence, $\lam^r_P$
is empty) set $\A_P=\A_0$ to be the prime alliance. The prime alliance
is special as it serves all invisible polynomials, however diverse they
are. It also serves all polynomials with non-repelling fixed points and
some other polynomials. By Section \ref{s:prime}, the prime alliance is
closed topologically and under friendship. This defines $\A_P$ for any
polynomial $P$ with $[P]\in\M_3$.

\begin{lem}
  \label{l:reg-all}
A regular alliance $\A_P$ has the form $\Cc_{P_0}$ for some visible
polynomial $P_0$, possibly different from $P$.
\end{lem}

\begin{proof}
Let $\lam$ be a chief of $\lam^r_P$. Then it has no infinite periodic gaps of degree 1.
Indeed, otherwise by Lemma \ref{l:crit-must} it has an
infinite 
gap 
with a critical edge.
By properties of laminations this edge is isolated, a contradiction with $\lam$ being perfect by Lemma \ref{l:oa-lam-1}.
By Lemma \ref{l:gland}, there is a
polynomial $P_0$ without neutral periodic points such that
$\lam^r_{P_0}=\lam$. Then by Lemma \ref{l:oa-lam-1}
$\A_P=\Cc(\lam)=\Cc_{P_0}$ as desired.
\end{proof}

A regular alliance is closed topologically (because $\Cc_{P_0}$ is
closed) and under friendship (by Corollary \ref{c:friends-regular}).

\begin{lem}
\label{l:CP-AP}
For any visible $P$ we have $\Cc_P\subset\A_P$.
\end{lem}

\begin{proof}
If $\lam^r_P$ is regular, $\Cc_P\subset\A_P$ by
definition. If $\lam^r_P$ is prime, then
$\A_P=\A_0$ is prime, and $\Cc_P\subset \A_P=\A_0$ since
$\A_0$ is closed under friendship.
\end{proof}

To prove the Main Theorem we need a couple of new concepts.
Let $P$ be such that $[P]\in \M_3$. A point $x$ is \emph{($P$-)stable}
if its forward orbit is finite and contains no critical or
non-repelling periodic points. The next lemma shows how stable points can be
applied.

\begin{lem}[\cite{hubbdoua85}, cf. Lemma B.1 \cite{gm93}]
 \label{l:rep}
Let $g$ be a polynomial, and $z$ be a stable point of $g$. If an
external ray $R_g(\theta)$ with rational argument $\theta$ lands at
$z$, then, for every polynomial $\tilde g$ sufficiently close to $g$,
the ray $R_{\tilde g}(\theta)$ lands at a stable point $\tilde z$ close
to $z$. Moreover, $\tilde z$ depends holomorphically on $\tilde g$.
\end{lem}

An unordered pair of rational angles
$\{\al,\be\}$ is \emph{($P$-)stable} if the external rays with
arguments $\al$ and $\be$ land at the same stable point of $P$.
Write $\Sc_P$ for the set of all $P$-stable pairs of angles; call $\Sc_P$ the \emph{s-set} of $P$.
Denote by $\Sc^{lam}_P$ the set of chords connecting $e^{2\pi i\al}$ with $e^{2\pi i \be}$,
 where the pair $\{\al,\be\}\in\Sc_P$ is not separated in $\R/\Z$ by any other $P$-stable pair of angles.
Evidently, $\Sc_P\subset \Zc_P$ and $\Sc^{lam}_P\subset \Zc^{lam}_P$; these sets do not
have to coincide as some legal points may be non-stable because their orbits
pass through critical points before they map to repelling periodic points.

If $\Sc_P\ne\0$, let $\Cc^s_P$ be the set of all critical portraits \emph{compatible}
with $\Sc_P$ (i.e., no critical chord from a critical portrait in
$\Cc^s_P$ crosses a leaf from $\Sc^{lam}_P$). It follows that $\Cc_P\subset \Cc^s_P$.

\begin{prop}
  \label{p:C-usc}
The dependence $P\mapsto \Cc^s_P$ is upper semi-continuous.
\end{prop}

\begin{proof}
We prove that if $P_i\to P$ and we choose $\K_i\in\Cc^s_{P_i}$ with
$\K_i\to\K$, then $\K\in\Cc^s_P$. Assume the contrary: $\K=\{\oc,\oy\}$,
where $\oc$ crosses some $\ell=\ol{ab}\in\Sc^{lam}_P$. By Lemma
\ref{l:rep}, $a\sim_{P_i} b$ for large $i$, and $\K_i$ contains
a critical chord $\oc_i$ close to $\oc$, a contradiction, since then
$\oc_i$ also crosses $\ell$.
\end{proof}

The next lemma relates $\Cc^s_P$ and the alliance $\A_P$.

\begin{lem}\label{l:stablegal}
Take $[P]\in \M_3$. If $\lam^r_P$ is regular, then $\Cc^s_P\subset \A_P$.
\end{lem}

\begin{proof}
Let $\lam$ be the chief of $\lam^r_P$. By Lemma \ref{l:oa-lam-1}, it 
is perfect.
Now, let
$\K$ be a critical portrait from $\Cc^s_P$ (i.e., compatible with $\Sc_P$).
We claim that then $\K$ is compatible with
$\lam$. Indeed, otherwise a critical leaf $\oc\in \K$ crosses a leaf $\ell\in \lam$.
By Lemma \ref{l:periodense},
arbitrarily Hausdorff-close to $\ell$,   there are iterated preimages of a periodic lap of $\lam$ that is not an eventual image of
a critical lap of $\lam$; it follows that edges of these preimages are leaves from $\Sc^{lam}_P$,
a contradiction with the assumption that $\K\subset \Cc^s_P$. Thus, $\Cc^s_P\subset \A_P$.
\end{proof}

We are ready to prove Theorem \ref{t:main} which implies the Main Theorem.

\begin{thm}
\label{t:main} The map $P\mapsto\A_P$ from $\M_3$ to the quotient space
of $\crp$ generated by alliances is continuous.
\end{thm}

For a critical portrait $\K$, consider the corresponding Thurston pullback lamination $\Tc(\K)$ (see \cite{thu85}).

\begin{lem}
  \label{l:thu}
  Let $G$ be a periodic lap of some invariant lamination, whose iterated images are disjoint from $\bigcup\K$.
Then leaves of $\Tc(\K)$ cannot cross edges of $G$.
\end{lem}

This is a straightforward corollary of \cite{thu85}.

\begin{proof}[Proof of Theorem \ref{t:main}]
Consider a sequence $P_i\to P$ of polynomials, and set $\A_i=\A_{P_i}$.
We need to show that sets $\A_i$ converge into $\A_P$. To this end it suffices
to consider critical portraits $\K_i\to\K$, where $\K_i\in\A_i$, and show that
$\K\in\A_P$. We may assume that either all $\K_i$
are prime, or all $\K_i$ are regular. If $\K_i$ are prime, then $\K$ is
prime by Lemma \ref{l:prime-cl}. Assume that all $\K_i$ are regular. By
Proposition \ref{p:C-usc}, we have $\K\in\Cc^s_P$. If $\lam^r_P$ is regular, then
by Lemma \ref{l:stablegal} $\K\in \A_P$, and we are done.

Assume that $\lam^r_P$ is prime. We need to show that $\K$ is prime.
Assume that $\K$ is regular. Then $\K$ is compatible with a regular
chief $\lam^\circ$.  By Lemma \ref{l:periodense} there are infinitely
many periodic 
laps of $\lam^\circ$ whose first return is onto.
Let $G$ be one of them, of period $n$, so that $\si^n(G)=G$.
We may assume that the (finite)
orbit of $G$ is disjoint from the chords in $\K$. 
Then the orbit of $G$ is disjoint from the chords in $\K_i$ with large $i$.

Set $\lam_i=\Tc(\K_i)$.
By Lemma \ref{l:thu}, the lap $G$ is contained in a lap $G_i$ of $\lam_i$.
Since $G$ is periodic, so is $G_i$.
The first return map of $\bd(G_i)$ has degree $1$ as any iterated image of $G_i$
 is contained in a complementary component of $\bigcup\K$.
Hence all periodic vertices of $G_i$ are of the same period.
In particular, it implies that if $G'$ and $G''$ are
periodic laps of $\lam^\circ$ such that periods of their vertices are
distinct, then the gap $G'_i$ cannot contain $G''_i$, and vice versa.

It is easy to see that an infinite periodic gap of degree 1 must have
at least one (pre-)critical edge, each critical edge serving at most two such gaps.
Thus, there may exist at most six cycles of infinite periodic degree 1 gaps of a cubic lamination.
Choose a periodic lap $G$ of $\lam^\circ$.
Passing to a subsequence, we may assume that all $G_i$ are finite, or all are infinite.
In the latter case, we \emph{reject} $G$ and replace it with another lap $G'$ whose vertex period is
greater than that of $G$.
By the above, there will be at most six rejected cycles of periodic gaps.
Refine our sequence of polynomials and choose
a sequence of finite periodic laps $G^1$, $\dots$, $G^j$, $\dots$ of $\lam^\circ$ such
that periods of their vertices grow and 
$G^j_i$ are finite for all $i$ and $j$.
If $\lam^\circ$ has Fatou gaps, we may assume that all $G^j_i$ are disjoint from them.


Set $G=G^j$. Since $\K_i$ is compatible with $\lam^r_{P_i}$, the
lamination $\lam^r_{P_i}$ consists of leaves that do not cross leaves
of $\lam_i$. Indeed, if a leaf $\ell'\in \lam^r_{P_i}$ crosses a leaf
$\ell\in \lam_i$, then we may assume that $\ell$ is a pullback of a
leaf of $\K_i$. Then, as above with $G$, the crossing of the two leaves
is kept until $\ell$ maps to a leaf of $\K_i$ which forces the
corresponding image of $\ell'$ to cross the same leaf of $\K_i$, a
contradiction. Hence leaves of $\lam^r_{P_i}$ do not cross leaves of
$\lam_i$.

All $G_i$ have the same period $n$.
By Kiwi \cite{kiw02}, vertices (edges) of gaps from the orbit of any finite
periodic gap of a cubic lamination form one or two orbits. Hence either
$G=G_i$, or $G\subsetneqq G_i$ in which case $G_i$ has two orbits of
edges. Since edges of $G_i$ do not cross leaves of $\lam^r_{P_i}$ then
there are at most finitely many leaves of $\lam^r_{P_i}$ intersecting
the interior of $G$; all these leaves are diagonals of $G_i$. It
follows that here are two cases listed below.

(1) A finite period $n$ lap $H_i$ of $\lam^r_{P_i}$ is non-disjoint from the interior of $G$.
Then $H_i$ is contained in $G_i$
(indeed, by construction the edges of $G_i$ are approximated by distinct pullbacks of leaves of $\K_i$
which are disjoint from leaves of $\lam^r_{P_i}$).

(2) An infinite period $\le n$ gap $U_i$ of $\lam^r_{P_i}$ contains $G_i\supset G$.

Accordingly, consider two cases.

(a) There are infinitely many gaps $G^j$ for which case (1) above
holds. We can find any number of distinct periodic laps which are
shared by all laminations $\lam^r_{P_i}$. Let $H$ be one of them chosen
so that all $P$-external rays whose arguments are vertices of $H$ land
at repelling periodic points. We claim that in fact all $P$-external
rays whose arguments are vertices of $H$ land at the same point.
Indeed, suppose otherwise. Then by Lemma \ref{l:rep} there are
$P_i$-external rays with vertices of $H$ as arguments that land at
distinct points, a contradiction. Thus, $H$ is (non-strictly) contained
in a finite lap $H'$ of $\lam^r_P$. We conclude that $\lam^r_P$ has
infinitely many finite periodic laps.

Thus, $\lam^r_P\ne \0$, and we can always choose a periodic lap $H'$ of
$\lam^r_P$ so that $H'$ corresponds to a point $z'$ which is not an
eventual image of a critical point of $P$. It follows that the entire
$P$-grand orbit of $z'$ consists of stable points. Consider the closure
$\lam'$ of the grand orbit of $H'$.
Evidently, $\lam'$ is a lamination, and $\lam'\subset \lam^r_P$.
If $\lam'$ is incompatible with $\K$, an iterated pullback $\ell$ of an edge of $H'$ crosses a chord in $\K$.
By Lemma \ref{l:rep}, there are leaves of $\lam^r_{P_i}$ that converge to $\ell$ as $i\to \infty$.
Since $\K_i\to \K$, for large $i$ some leaves of $\lam^r_{P_i}$ cross chords in $\K_i$, a contradiction.
Hence $\lam'$ is
compatible with $\K$. However, since $\K$ is regular, it implies that
$\lam'$ and $\lam^r_P\supset \lam'$ are regular, a contradiction.

(b) For all but finitely many gaps $G^j$ and all but finitely many $i$
case (2) above holds and a periodic Fatou gap $U^j_i\supset G^j$ of
period at most $n_j$ exists. Moreover, since there are infinitely many
gaps $G^j$, Fatou gaps in question will contain infinitely many gaps
$G^j$. A priori, as $i\to \infty$, these gaps may change.
We may assume that $\lam^r_{P_i}\to\lam''$.
Since $\K_i\to \K$ and $\K_i$'s are
compatible with $\lam^r_{P_i}$'s, it follows that $\lam''$ is
compatible with $\K$. Hence $\lam^\circ\subset \lam''$. In particular,
the limit $U^j$ of the gaps $U^j_i$ must be contained in a Fatou gap of
$\lam^\circ$ while containing $G^j$. However no gap $G^j$ can be
contained in a Fatou gap of $\lam^\circ$, a contradiction.
\end{proof}


\begin{thebibliography}{99}

\bibitem{bfmot10} A. Blokh, R. Fokkink, J. Mayer, L. Oversteegen, E.
    Tymchatyn, \emph{Fixed point theorems in plane continua with
    applications}, Memoirs of the American Mathematical Society,
    \textbf{224} (2013), no. 1053.

\bibitem{bl02} A. Blokh, G. Levin, \emph{Growing trees,
    laminations and the dynamics on the Julia set}, Erg. Th. and Dyn.
Syst. \textbf{22} (2002), 63--97.

\bibitem{bmov13} A. Blokh, D. Mimbs, L. Oversteegen, K.
    Valkenburg, \emph{Laminations in the language of leaves}, Trans.
    Amer. Math. Soc., \textbf{365} (2013), 5367--5391.

\bibitem{bo04}  A. Blokh, L. Oversteegen, \emph{Wandering
    triangles exist}, C. R. Math. (Acad. Sci., Paris), \textbf{339} (2004), 365--370.

\bibitem{bowaga} A. Blokh, L. Oversteegen, \emph{Wandering gaps for
    weakly
    hyperbolic cubic polynomials}, in: ``Complex dynamics: families and friends'',
    ed. by D. Schleicher, A K Peters (2008), ISBN 13978-1-56881-450-6,
    121--150.

\bibitem{bopt13} A. Blokh, L. Oversteegen, R. Ptacek, V. Timorin, \emph{
Dynamical cores of topological polynomials,}
Proceedings for the Conference ``Frontiers in Complex Dynamics (Celebrating John Milnor's 80th birthday)'',
Princeton Mathematical Series \textbf{51}, Princeton University Press (2013), 27-47.

\bibitem{bopt14} A. Blokh, L. Oversteegen, R. Ptacek, V.
   Timorin, \emph{The Main Cubioid}, Nonlinearity, \textbf{27} (2014),
    1879--1897.

\bibitem{bopt16a} A. Blokh, L. Oversteegen, R. Ptacek, V.
   Timorin, \emph{Laminations from the Main Cubioid}, Discrete and Continuous
   Dynamical Systems \textbf{36} (2016), 4665--4702.

\bibitem{bopt15} A. Blokh, L. Oversteegen, R. Ptacek, V.
    Timorin, \emph{The combinatorial Mandelbrot set as the quotient of the space of
    geolaminations}, Contemporary Mathematics \textbf{669} (2016),
    37--62.

\bibitem{bopt18} A. Blokh, L. Oversteegen, R. Ptacek, V.
   Timorin, \emph{Complementary components to the cubic principal domain},
   Proc. Amer. Math. Soc. ,
   \textbf{146} (2018), 4649--4660.

\bibitem{bopt20} A. Blokh, L. Oversteegen, R. Ptacek, V.
    Timorin, \emph{Laminational models for spaces of polynomials of any
    degree}, Memoirs of the AMS \textbf{265} (2020), no. 1288.

\bibitem{bh01} X. Buff, C. Henriksen, \emph{Julia Sets in
    Parameter Spaces}, Commun. Math. Phys. \textbf{220} (2001), 333--375.

\bibitem{dh82} A. Douady, J. H. Hubbard, \emph{It\'eration des
    polyn\^omes quadratiques complexes},  C. R. Acad. Sci. Paris S\'er.
    I Math. \textbf{294} (1982), no. 3, 123--126.

\bibitem{hubbdoua85} A. Douady, J. H. Hubbard,  \emph{\'Etude dynamique
    des polyn\^omes complexes} I, II, Publications Math\'ematiques
    d'Orsay \textbf{84-02} (1984), \textbf{85-04} (1985).

\bibitem{gm93} L. Goldberg, J. Milnor, \emph{Fixed points of
   polynomial maps. Part II. Fixed point portraits}, Annales
    scientifiques de l'ENS (4) \textbf{26} (1993), 51--98.

\bibitem{kiw01} J. Kiwi, \emph{Rational laminations of complex polynomials},
In: Laminations and foliations in dynamics, geometry and topology
(Stony Brook, NY, 1998) Contemp. Math. \textbf{269} (2001), Amer. Math. Soc.,
Providence, RI, 111-154.

\bibitem{kiw02} J. Kiwi, \emph{Wandering orbit portraits},
  Trans. Amer. Math. Soc. \textbf{354} (2002), 1473--1485.

\bibitem{kiwi97} J. Kiwi, \emph{$\mathbb R$eal laminations and
    the topological dynamics of complex polynomials}, Advances in
    Mathematics, \textbf{184} (2004), 207--267.

\bibitem{la89} P. Lavaurs, \emph{Syst\`emes dynamiques
    holomorphes, Explosion de points p\'eriodiques paraboliques},
   Th\`ese de doctorat, Universit\'e Paris-Sud, Orsay (1989).

\bibitem{thu85} W.~Thurston, \newblock {\em The combinatorics of
    iterated rational maps} (1985), published in: ``Complex dynamics:
    Families and Friends'', ed. by D. Schleicher, A K Peters (2009),
    1--108.

\bibitem{tby20} W. Thurston, H. Baik, G. Yan, J. Hubbard, L. Tan, K.
    Lindsey, D. Thurston, \emph{Degree $d$ invariant laminations}, in:
    ``What's Next?: The Mathematical Legacy of William P. Thurston'',
    ed. by D. Thurston, AMS -- 205, Princeton U. Press (2020).

\end{thebibliography}
\end{document}